\documentclass[12pt, a4paper]{article}

\usepackage[utf8]{inputenc}
\usepackage{amsmath}
\usepackage{amsfonts}
\usepackage{amssymb}
\usepackage{amsthm}
\usepackage{graphicx}
\usepackage{hyperref}
\usepackage{geometry}

\newtheorem{theorem}{Theorem}

\title{Perimeter of an Ellipse: Understanding Ramanujan's Approximation}
\author{Uday Shankar}
\date{}

\begin{document}

\maketitle

\section{Introduction}

Srinivasa Ramanujan \cite{ramanujan1962collected_standard} mentions two approximations to the perimeter of an ellipse that are amazingly accurate. However, he does not provide an explanation of how he arrived at those expressions. In this paper, we will try to provide such an explanation. We start with a (possibly) well known derivation of the exact perimeter of an ellipse as an infinite series using the arc length formula. This infinite series can serve as the starting point for our explanation. We convert this series into its continued fraction expansion and show that Ramanujan's formula can be derived from different approximations of the continued fractions. Following the derivation, we improve these approximations using two strategies: (1) we apply a correction using a slight perturbation to Ramanujan's second approximation to get an alternate expression and (2) we extend the idea used by Ramanujan to approximate the tail coefficients of the continued fraction better. These expressions, although not as elegant as Ramanujan's, give a better approximation to the perimeter of an ellipse. Villarino ~\cite{villarino2005ramanujansperimeterellipse} presented a detailed error analysis to Ramanujan's second approximation. Cantrell ~\cite{Cantrell2004} provides a different approximation that is worse than Ramanujan's expression for ellipses with lower eccentricity, but is better when the eccentricity gets closer to 1. However, to the best of our knowledge, ours is the first attempt to explain Ramanujan's ellipse perimeter formula and our approximation is uniformly better than his expression.

\section{Exact Derivation}

In this section, we will derive the exact form for the perimeter of the ellipse using well-known techniques from calculus and complex analysis. We will start by introducing some basic identities which we will use later in the paper. 

\subsection{Basic Identities}

The extension of the real cosine function to the complex plane, denoted as $\cos(z)$ for complex number $z$, is defined based on Euler's formula, $e^{iz} = \cos(z) + i\sin(z)$. By considering the complementary equation $e^{-iz} = \cos(z) - i\sin(z)$ and adding the two, we arrive at the definition:
\begin{equation}
    \cos(z) = \frac{e^{iz} + e^{-iz}}{2}
\end{equation}

\noindent Next, we derive the infinite series expansion for the function $f(x) = (1-x)^{1/2}$. We utilize the generalized binomial theorem, which states that for any real number $\alpha$, $(1+u)^\alpha = \sum_{n=0}^{\infty} \binom{\alpha}{n} u^n$ for $|u| < 1$, where the generalized binomial coefficient is defined as $\binom{\alpha}{n} = \frac{\alpha(\alpha-1)\dots(\alpha-n+1)}{n!}$. For our function, we set $u = -x$ and $\alpha = \frac{1}{2}$. The expansion is
\begin{align}
    (1-x)^{1/2} &= \sum_{n=0}^{\infty} \binom{1/2}{n} (-x)^n \\
    &= 1 + \sum_{n=1}^{\infty} (-1)^n \frac{(1/2)(-1/2)(-3/2)\cdots(3-2n)/2}{n!} x^n \\
    &= 1 + \sum_{n=1}^{\infty} \frac{-1.1.3.5\cdots(2n-3)}{2^n n!} x^n \\
    &= \sum_{n=0}^{\infty} \frac{-1.1.3.5\cdots(2n-3)}{2^n n!} x^n
\end{align}

\noindent The first few coefficients of the above series is
\begin{align*}
    n=0: \quad \binom{1/2}{0}(-x)^0 &= 1 \cdot 1 = 1 \\
    n=1: \quad \binom{1/2}{1}(-x)^1 &= \left(\frac{1}{2}\right)(-x) = -\frac{1}{2}x = \frac{-1}{2^1 ~1!}x\\
    n=2: \quad \binom{1/2}{2}(-x)^2 &= \frac{(\frac{1}{2})(\frac{1}{2}-1)}{2!}x^2 = \frac{(\frac{1}{2})(-\frac{1}{2})}{2}x^2 = -\frac{1}{8}x^2 = \frac{-1.1}{2^2 ~2!}x^2\\
    n=3: \quad \binom{1/2}{3}(-x)^3 &= \frac{(\frac{1}{2})(-\frac{1}{2})(-\frac{3}{2})}{3!}(-x^3) = \frac{\frac{3}{8}}{6}(-x^3) = -\frac{1}{16}x^3 = \frac{-1.1.3}{2^3 ~3!}x^3\\
    n=4: \quad \binom{1/2}{4}(-x)^4 &= \frac{(\frac{1}{2})(-\frac{1}{2})(-\frac{3}{2})(-\frac{5}{2})}{4!}x^4 = \frac{-\frac{15}{16}}{24}x^4 = -\frac{5}{128}x^4 = \frac{-1.1.3.5}{2^4 ~4!}x^4
\end{align*}
Substituting these back into the summation yields the final series expansion, valid for $|x| < 1$,
\begin{equation}
    (1-x)^{1/2} = 1 - \frac{1}{2}x - \frac{1}{8}x^2 - \frac{1}{16}x^3 - \frac{5}{128}x^4 - \dots
\end{equation}

\subsection{Perimeter of Ellipse using the Arc Length Formula}

An ellipse centered at the origin with semi-major axis $a$ and semi-minor axis $b$ (assuming $a > b > 0$) can be described by the parametric equations:
\begin{align*}
    x(t) &= a \cos t \\
    y(t) &= b \sin t
\end{align*}
for $t \in [0, 2\pi]$. The arc length $S$ of a parametric curve is given by the integral:
\[ S = \int_{t_1}^{t_2} \sqrt{\left(\frac{dx}{dt}\right)^2 + \left(\frac{dy}{dt}\right)^2} \, dt \]

\noindent The derivatives in the arc length expression are
\begin{align*}
    \frac{dx}{dt} &= -a \sin t \\
    \frac{dy}{dt} &= b \cos t
\end{align*}

\noindent Substituting these into the arc length formula for the full perimeter $P$ gives
\[ P = \int_0^{2\pi} \sqrt{(-a \sin t)^2 + (b \cos t)^2} \, dt = \int_0^{2\pi} \sqrt{a^2 \sin^2 t + b^2 \cos^2 t} \, dt \]

\noindent The integrand $a^2 \sin^2 t + b^2 \cos^2 t$ can be written in two ways: $a^2 + (b^2 - a^2) \cos^2 t$ or $b^2 + (a^2 - b^2) \sin^2 t$. Adding the two expressions and dividing by 2 gives us 

\begin{align}
a^2 \sin^2 t + b^2 \cos^2 t &= \frac{1}{2} ((a^2 + b^2) + (b^2 - a^2)(\cos^2 t - \sin^2 t)) \\
&= \frac{1}{2} \left( \frac{(a+b)^2}{2} + \frac{(b-a)^2}{2} + \frac{2}{2}(b^2 - a^2) \cos 2t \right) \\
&= \frac{(a+b)^2}{4} \left( 1 + \left(\frac{b-a}{a+b}\right)^2 - 2 \left(\frac{a-b}{a+b}\right) \cos 2t \right) \\
&= \left( \frac{a+b}{2} \right)^2 (u^2 - 2u \cos 2t + 1) \\
&= \left( \frac{a+b}{2} \right)^2 (u^2 - u (e^{2it} + e^{-2it}) + 1) \\
&= \left( \frac{a+b}{2} \right)^2 (u - e^{2it})(u - e^{-2it}) \\
&= \left( \frac{a+b}{2} \right)^2 (1 - ue^{2it})(1 - ue^{-2it}) 
\end{align}

\noindent where $u = \frac{a-b}{a+b}$ and we replace the cosine function with its complex extension. Substituting this new expression back into the formula for the perimeter, we get

\begin{align}
    P &= 
    (a+b) \int_0^{\pi} \sqrt{(1 - ue^{2it})(1 - ue^{-2it})} \, dt \label{eq:perim1}
\end{align}

\noindent Now we prove the following result attributed to James Ivory in the late $18^{th}$ century.

\begin{theorem}
If $0 \leqslant u \leqslant 1$ then 
\begin{equation} 
\frac{1}{\pi} \int_{0}^{\pi} \sqrt{1 - ue^{2it}} \sqrt{1 - ue^{-2it}} \, dt = \sum_{n=0}^{\infty} \left\{ \frac{-1.1.3.5\cdots(2n-3)}{2^n n!} \right\}^2 u^{2n} 
\end{equation}    
\end{theorem}

\begin{proof} The proof uses our derivation for the series expansion of $(1-x)^{1/2}$. 
 \begin{align*}
\frac{1}{\pi} &\int_{0}^{\pi} \sqrt{1 - ue^{2it}} \sqrt{1 - ue^{-2it}} \, dt \\
&= \frac{1}{\pi} \int_{0}^{\pi} \sum_{m=0}^{\infty} \left\{ \frac{-1.1.3.5\cdots(2m-3)}{2^m m!} \right\} u^m e^{2imt} \sum_{n=0}^{\infty} \left\{ \frac{-1.1.3.5\cdots(2n-3)}{2^n n!} \right\} u^n e^{-2int} \, dt \\
&= \frac{1}{\pi} \sum_{m=0}^{\infty} \left\{ \frac{-1.1.3.5\cdots(2m-3)}{2^m m!} \right\} u^m \sum_{n=0}^{\infty} \left\{ \frac{-1.1.3.5\cdots(2n-3)}{2^n n!} \right\} u^n \int_{0}^{\pi} e^{2\pi i(m-n)t} \, dt \\
&= \sum_{n=0}^{\infty} \left\{ \frac{-1.1.3.5\cdots(2n-3)}{2^n n!} \right\}^2 u^{2n}
\end{align*} 

\noindent The final step can be obtained by noticing that the integral is non-zero only when $m = n$.  
\end{proof}

\noindent We can now substitute the result of the above theorem into equation ~\ref{eq:perim1} to get

\begin{align}
    P &= \pi (a+b) \left( \sum_{n=0}^{\infty} \left\{ \frac{-1.1.3.5\cdots(2n-3)}{2^n n!} \right\}^2 u^{2n} \right) \\
    &= \pi (a+b) ~E(x)
\end{align}

\noindent where $x = u^2 = \left( \frac{a-b}{a+b} \right)^2$ and 

\begin{align*}
    E(x) &= \sum_{n=0}^\infty \left( \frac{-1.1.3.5\cdots(2n-3)}{2^n n!} \right)^2 x^n \\
    &= \sum_{n=0}^\infty \left( \frac{1.3.5\cdots(2n-3)}{2^n n!} \right)^2 x^n \\
    &= \sum_{n=0}^\infty \left( \frac{1.2.3.4.5\cdots(2n-3)(2n-2)}{2^{n-1} (n-1)! 2^n n!} \right)^2 x^n \\
    &= \sum_{n=0}^\infty \left( \frac{1.2.3.4.5\cdots(2n-3)(2n-2)(2n-1)(2n)}{(2n-1) 2n 2^{n-1} (n-1)! 2^n n!} \right)^2 x^n \\
    &= \sum_{n=0}^\infty \left( \frac{(2n)!}{(2n-1) 2^n n! 2^n n!} \right)^2 x^n \\
    &= \sum_{n=0}^\infty \left( \frac{1}{(2n-1) 4^n} \binom{2n}{n}\right)^2 x^n \\
\end{align*}

\noindent By forming different approximations of the function $E(x)$, we can derive simpler expressions for the perimeter of the ellipse.

\section{Derivation of Ramanujan's Approximations via Continued Fractions \label{sec:sec3}}

In his collected works, 
Srinivasa Ramanujan \cite{ramanujan1962collected_standard} provided two approximations for the perimeter of an allipse. He mentions that he arrived at these expressions empirically, although he never explained his approach. The first expression is

\begin{align*}
    P &\approx \pi (3(a+b) - \sqrt{(3a+b)(a+3b)}) \\
    P_1 &= \pi (a+b) \left( 3 - \sqrt{\frac{(3a+b)(a+3b)}{(a+b)(a+b)}} \right) \\
    &= \pi (a+b) \left( 3 - \sqrt{(2 + \frac{a-b}{a+b})(2 - \frac{a-b}{a+b})} \right) \\
    &= \pi (a+b) (3 - \sqrt{(2+u)(2-u)}) \\
    &= \pi (a+b) (3 - \sqrt{4 - u^2}) \\
    &= \pi (a+b) (3 - \sqrt{4 - x}) \\
    &= \pi (a+b) R_1(x)
\end{align*}

\noindent His second approximation had the form

\begin{align*}
    P &\approx \pi \left( (a + b) + \frac{3(a-b)^2}{10(a+b) + \sqrt{a^2 + 14ab + b^2}}\right) \\
\end{align*}

\noindent which can be written as 

\begin{align*}
    P_2 &= \pi (a+b) \left( 1 + \frac{3(a-b)^2}{(a+b)^2 (10 + \sqrt{\frac{a^2 + 14ab + b^2}{(a+b)^2}})} \right) \\
    &= \pi (a+b) \left( 1 + \frac{\frac{3(a-b)^2}{(a+b)^2}}{10 + \sqrt{4 - 3\frac{(a-b)^2}{(a+b)^2}}} \right) \\
    &= \pi (a+b) \left( 1 + \frac{3x}{10+\sqrt{4-3x}} \right) \\
    &= \pi (a+b) R_2(x)
\end{align*}

\noindent The remarkable accuracy of Ramanujan's approximations, $R_1(x)$ and $R_2(x)$ can be understood by examining the continued fraction expansion of the target power series, $E(x)$. The power series is given by:

\begin{equation}
    E(x) = 1 + \frac{1}{4}x + \frac{1}{64}x^2 + \frac{1}{256}x^3 + \frac{25}{16384}x^4 + \frac{49}{65536}x^5 + \dots
\end{equation}

\section{Continued Fraction Expansion of E(x)}

The transformation of the power series $E(x)$ into a generalized continued fraction is achieved through a systematic process of series inversion known as Viscovatov's algorithm \cite{baker1996pade}. Starting with the power series representation:

\begin{equation}
    E(x) = 1 + \frac{1}{4}x + \frac{1}{64}x^2 + \frac{1}{256}x^3 + \dots
\end{equation}
We seek a continued fraction of the form:
\begin{equation}
    E(x) = 1 + \cfrac{a_1 x}{1 + \cfrac{a_2 x}{1 + \cfrac{a_3 x}{1 + \dots}}}
\end{equation}

\noindent The algorithm determines the partial numerators $\{a_n\}$ iteratively using an "invert and subtract" method based on the identity $Y = 1 + \frac{a x}{ (ax / (Y-1)) }$.

\subsection{\label{ssec:continued_fraction} High-Level Algorithm Summary}
Let $S_0(x) = E(x)$. The iteration proceeds as follows for $n \ge 1$:
\begin{enumerate}
    \item \textbf{Identify Coefficient:} The coefficient $a_n$ is determined by taking the coefficient of the linear term $x$ in the current series $S_{n-1}(x)$.
    \item \textbf{Form Remainder:} A new fractional expression is formed: $\frac{a_n x}{S_{n-1}(x) - 1}$.
    \item \textbf{Invert and Expand:} This fraction is expanded via series inversion to generate the next power series, $S_n(x)$, which must start with a constant term of 1.
    \item \textbf{Repeat:} The process is repeated on $S_n(x)$ to find $a_{n+1}$.
\end{enumerate}

\noindent We determine the coefficients $a_n$ sequentially by computing the series expansion of the reciprocal remainders.
The power series for $E(x)$ begins:
\[
E(x) = 1 + \frac{1}{4}x + \frac{1}{64}x^2 + \frac{1}{256}x^3 + \dots
\]
The first coefficient $a_1$ is simply the coefficient of $x$ in $E(x)$.
\[
a_1 = \frac{1}{4}
\]
We then define the first remainder function $E_1(x)$ such that $E(x) = 1 + \frac{a_1 x}{E_1(x)}$. Rearranging for $E_1(x)$,
\[
E_1(x) = \frac{a_1 x}{E(x) - 1} = \frac{\frac{1}{4}x}{\frac{1}{4}x + \frac{1}{64}x^2 + \frac{1}{256}x^3 + \dots}
\]
Dividing numerator and denominator by $\frac{1}{4}x$, we get
\[
E_1(x) = \frac{1}{1 + \frac{1}{16}x + \frac{1}{64}x^2 + \dots}
\]
Computing the series expansion of this reciprocal (using $(1+u)^{-1} \approx 1-u+u^2$):
\[
E_1(x) = 1 - \frac{1}{16}x - \frac{3}{256}x^2 + O(x^3)
\]

\noindent 
The continued fraction definition implies $E_1(x) = 1 + \frac{a_2 x}{E_2(x)}$. Thus, $a_2$ is the coefficient of the linear term in $E_1(x)$:
\[
a_2 = -\frac{1}{16}
\]
We now solve for the next remainder $E_2(x) = \frac{a_2 x}{E_1(x) - 1}$:
\[
E_2(x) = \frac{-\frac{1}{16}x}{-\frac{1}{16}x - \frac{3}{256}x^2 + \dots} = \frac{-\frac{1}{16}}{-\frac{1}{16} - \frac{3}{256}x + \dots}
\]
Dividing through by $-\frac{1}{16}$:
\[
E_2(x) = \frac{1}{1 + \frac{3}{16}x + \dots} = 1 - \frac{3}{16}x + O(x^2)
\]

\noindent 
Similarly, $a_3$ is the linear coefficient of $E_2(x)$:
\[
a_3 = -\frac{3}{16}
\]

\noindent 
Repeating this procedure yields the subsequent coefficients:
\[
\begin{aligned}
a_4 &= -\frac{3}{16} \\
a_5 &= -\frac{11}{48} \\
a_6 &= -\frac{115}{528} \\
a_7 &= -\frac{14627}{60720} \\
a_8 &= -\frac{18556241}{80741040} 
\end{aligned}
\]

\noindent Applying this algorithm to $E(x)$ yields the following sequence of partial numerators
\begin{equation}
    a_1 = \frac{1}{4}, \quad a_2 = -\frac{1}{16}, \quad a_3 = -\frac{3}{16}, \quad a_4 = -\frac{3}{16}, \quad a_5 = -\frac{11}{48}, \dots
\end{equation}

\noindent We can construct the simplified integer form of the expansion by multiplying the numerator and denominator of successive fraction layers by 4 to clear the denominators.

\begin{equation}
    \label{eq:Bx_CF}
    E(x) = 1 + \cfrac{x}{4 - \cfrac{x}{4 - \cfrac{3x}{4 - \cfrac{3x}{4 - \cfrac{\frac{11}{3}x}{4 - \dots}}}}}
\end{equation}

\subsection{Derivation of $R_1(x)$}

Ramanujan's first approximation is given by $R_1(x) = 3 - \sqrt{4-x}$. To see its connection to the continued fraction, we observe the very first layer of $E(x)$ in Equation \ref{eq:Bx_CF}: $1 + \frac{x}{4 - \dots}$.

\noindent If we assume the simplest possible periodic pattern based on this initial layer — that the partial numerator is always $-x$ and the denominator is always $4 - \dots$ — we define an approximation $R_{cf}(x)$:
\begin{equation}
    R_{cf}(x) = 1 + \cfrac{x}{4 - \cfrac{x}{4 - \cfrac{x}{4 - \dots}}}
\end{equation}
Let the repeating fraction part be $Y$. Then $Y = \frac{x}{4-Y}$. Solving the quadratic equation $Y^2 - 4Y + x = 0$ and choosing the solution that vanishes at $x=0$ yields $Y = 2 - \sqrt{4-x}$. Substituting this back gives:
\begin{equation}
    R_{cf}(x) = 1 + (2 - \sqrt{4-x}) = 3 - \sqrt{4-x}
\end{equation}
This is exactly Ramanujan's formula $R_1(x)$. Its power series expansion matches $E(x)$ for the first three terms ($1, x, x^2$).

\subsection{Derivation of $R_2(x)$}

Ramanujan's more accurate approximation is $R_2(x) = 1 + \frac{3x}{10 + \sqrt{4-3x}}$. This formula arises from matching the continued fraction of $E(x)$ deeper into its expansion. Observing Equation \ref{eq:Bx_CF}, the sequence of partial numerators begins $x, -x, -3x, -3x$. If we assume that after the initial $x, -x$, the pattern of repeating pairs of $-3x$ continues infinitely, we define the approximation $R_{cf}(x)$ as follows:

\begin{equation}
    R_{cf}(x) = 1 + \cfrac{x}{4 - \cfrac{x}{4 - \underbrace{\cfrac{3x}{4 - \cfrac{3x}{4 - \dots}}}_{P}}}
\end{equation}
We isolate the repeating tail $P$. It satisfies the recursive relation $P = \frac{3x}{4-P}$. Solving $P^2 - 4P + 3x = 0$ yields $P = 2 - \sqrt{4-3x}$. Substituting $P$ back into the previous layer of the fraction gives us
\begin{equation}
    4 - P = 4 - (2 - \sqrt{4-3x}) = 2 + \sqrt{4-3x}
\end{equation}
We continue by substituting this result back into the layer above it.
\begin{equation}
    4 - \frac{x}{2 + \sqrt{4-3x}} = \frac{4(2+\sqrt{4-3x}) - x}{2 + \sqrt{4-3x}} = \frac{8 - x + 4\sqrt{4-3x}}{2 + \sqrt{4-3x}}
\end{equation}
Finally, substituting this into the main expression for $R_{cf}(x)$ results in
\begin{align}
    R_{cf}(x) &= 1 + \frac{x(2 + \sqrt{4-3x})}{8 - x + 4\sqrt{4-3x}} \\
    &= 1 + \frac{x(2 + \sqrt{4-3x})(2 - \sqrt{4-3x})}{(8 - x + 4\sqrt{4-3x})(2 - \sqrt{4-3x})} \\
    &= 1 + \frac{3x^2}{2(8-x) - 4(4-3x) + (8 - (8-x))\sqrt{4-3x}} \\
    &= 1 + \frac{3x^2}{10x + x\sqrt{4-3x}}
\end{align}
This leads to the known compact form,
\begin{equation}
    R_2(x) = 1 + \frac{3x}{10 + \sqrt{4-3x}}
\end{equation}
Because its continued fraction structure matches $E(x)$ through the pair of $-3x$ terms, its power series expansion perfectly matches $E(x)$ for the first five terms (up to $x^4$). In summary, Ramanujan's formulas are exact closed-form representations of infinite periodic continued fractions that serve as increasingly accurate asymptotic approximations for the non-periodic continued fraction of $E(x)$. Unfortunately, there is no simple extension of this technique to obtain an improvement over $R_2(x)$.

\section{Improving on Ramanujan's Approximation - Attempt 1}

\subsection{Motivation}
Our starting point is Ramanujan's approximation $R_2(x)$, derived from the continued fraction expansion of the target series $E(x)$.
\[ R_2(x) = 1 + \frac{3x}{10 + \sqrt{4 - 3x}} \]
We established that $R_2(x)$ perfectly matches the series expansion of $E(x)$ for the first 5 terms (up to $x^4$). The deviation begins at the $6^{th}$ term ($x^5$). Our goal is to construct a new approximation, $A_1(x)$, that retains the  advantages of the square root form in $R_2(x)$ but extends the accuracy to match the $x^5$ term of $E(x)$. To achieve this, we employ a \textit{perturbation method}. We will introduce a small correction term inside the highly-nested square root of $R_2(x)$. This placement ensures that the correction does not disrupt the match of the lower-order terms ($x^0$ to $x^4$) while providing a parameter to tune the $x^5$ coefficient precisely.

\subsection{Required Correction}

The target series $E(x)$ and $R_2(x)$ can be written as

\begin{align}
    E(x) &\approx 1 + \frac{1}{4}x + \frac{1}{64}x^2 + \frac{1}{256}x^3 + \frac{25}{16384}x^4 + \frac{49}{65536}x^5 + \frac{441}{1048576}x^6 + \frac{1089}{4194304}x^7 + \dots \\
    R_2(x) &\approx 1 + \frac{1}{4}x + \frac{1}{64}x^2 + \frac{1}{256}x^3 + \frac{25}{16384}x^4 + \frac{95}{131072}x^5 + \frac{2571}{8388608}x^6 + \frac{10573}{67108864}x^7 + \dots
\end{align}
First, we determine the exact discrepancy between the target series $E(x)$ and the approximation $R_2(x)$ at the $x^5$ term. Let $a_n^f$ denote the coefficient of the $x^n$ term in the expansion of $f(x)$. Then

\begin{align*}
    a_5^E &= \frac{49}{65536} = \frac{98}{131072} \\
    a_5^{R_2} &= \frac{95}{131072}
\end{align*}
The required correction, $\Delta$, needed to bring the coefficient up to the target value is:
\begin{equation}
    \Delta = a_5^E - a_5^{R_2} = \frac{98 - 95}{131072} = \frac{3}{131072}
\end{equation}

\subsection{Defining the Form of Perturbation}
The structure of $R_2(x)$ is $1 + 3x \cdot D(x)^{-1}$, where $D(x) = 10 + \sqrt{4-3x}$. To affect the $x^5$ term of the final expression, we must introduce a perturbation of order $x^4$ inside the denominator, which then gets multiplied by the outer $3x$.\\

\noindent We propose the perturbed form $A_1(x)$ by adding a term $kx^4$ inside the radical expression.
\begin{equation}
    A_1(x) = 1 + \frac{3x}{10 + \sqrt{4 - 3x + kx^4}}
\end{equation}
Here, $k$ is the unknown constant we must determine.
We now linearize the effect of the small perturbation $kx^4$ on the expansion coefficients. Let $D_{R_2}(x) = 10 + \sqrt{4-3x}$ be the denominator in the original approximation.
Let $D_{A_1}(x) = 10 + \sqrt{4 - 3x + kx^4}$ be the denominator in the new approximation.\\

\noindent We can approximate the change in the square root term using the first derivative of $\sqrt{u}$ with respect to $u$: $\delta\sqrt{u} \approx \frac{1}{2\sqrt{u}} \delta u$. In our case, 
$u$ is $4-3x$ and the perturbation $\delta u$ is $kx^4$.
\[ \sqrt{4 - 3x + kx^4} \approx \sqrt{4-3x} + \frac{kx^4}{2\sqrt{4-3x}} \]
Therefore, the denominator $D_{A_1}(x)$ can be approximated as
\[ D_{A_1}(x) \approx D_{R_2}(x) + \frac{kx^4}{2\sqrt{4-3x}} \]
Next, we determine the effect on the reciprocal of the denominator. Using the approximation $\frac{1}{D+\delta} \approx \frac{1}{D} - \frac{\delta}{D^2}$,
\[ \frac{1}{D_{A_1}(x)} \approx \frac{1}{D_{R_2}(x)} - \frac{\left( \frac{kx^4}{2\sqrt{4-3x}} \right)}{D_{R_2}(x)^2} = \frac{1}{D_{R_2}(x)} - \frac{kx^4}{2\sqrt{4-3x} \cdot D_{R_2}(x)^2} \]
We want to find the change to the $x^4$ coefficient of this reciprocal. The second term on the RHS is already of order $x^4$. To find its contribution to the $x^4$ coefficient, we evaluate the rest of that term at $x=0$.
\[ \text{Change to } a_4^D \approx - \frac{k}{2\sqrt{4-0} \cdot D_{R_2}(0)^2} \]
Noting that $D_{R_2}(0) = 10 + \sqrt{4} = 12$,
\[ \text{Change to } a_4^D \approx - \frac{k}{2(2)(12^2)} = -\frac{k}{576} \]
Finally, the expression $A_1(x)$ multiplies this reciprocal by $3x$. Therefore, the change to the $x^5$ coefficient of $A_1(x)$ is $3$ times the change calculated above.
\[ \text{Change to } a_5^{A_1} = 3 \cdot \left( -\frac{k}{576} \right) = -\frac{3k}{576} = -\frac{k}{192} \]

\noindent We set this calculated change equal to the required correction $\Delta$.

\begin{align}
    -\frac{k}{192} &= \frac{3}{131072} \\
    k &= -192 \cdot \frac{3}{131072} \\
    &= -\frac{576}{131072} \\
    &= -\frac{9}{2048}
\end{align}

\noindent Substituting this value of $k$ back into our proposed form yields the approximation $A_1(x)$ which perfectly matches the first 6 terms of $E(x)$,
\begin{align}
    A_1(x) &= 1 + \frac{3x}{10 + \sqrt{4 - 3x - \frac{9}{2048}x^4}} \\
    &= 1 + \frac{64 \cdot 3x}{64 \left( 10 + \sqrt{4 - 3x - \frac{9}{2048}x^4} \right)} \\
    &= 1 + \frac{192x}{640 + 64\sqrt{4 - 3x - \frac{9}{2048}x^4}} \\
    &= 1 + \frac{192x}{640 + \sqrt{16384 - 12288x - 18x^4}}
\end{align}

\noindent Thus, the final approximation is
\[
\boxed{A_{1}(x) = 1 + \frac{192x}{640 + \sqrt{16384 - 12288x - 18x^4}}}
\]

\section{Can we do Better? - Attempt 2}

The continued fraction expansion of $E(x)$ was derived in section~\ref{ssec:continued_fraction}. We found the coefficients $a_i$ to be

\[
\begin{aligned}
a_1 &= -\frac{1}{4} = 0.25 \\
a_2 &= -\frac{1}{16} = -0.0625 \\
a_3 &= -\frac{3}{16} = -0.1875 \\
a_4 &= -\frac{3}{16} = -0.1875 \\
a_5 &= -\frac{11}{48} \approx -0.2292 \\
a_6 &= -\frac{115}{528} \approx -0.2178 \\
a_7 &= -\frac{14627}{60720} \approx -0.2409 \\
a_8 &= -\frac{18556241}{80741040} \approx -0.2298
\end{aligned}
\] \\

\noindent Analyzing the coefficients for $n \ge 5$, we observe they oscillate around $-0.2288$. We construct an approximation by assuming the tail stabilizes immediately after $a_4$ with a constant coefficient:
\[
a_n = -0.2288 \approx -\frac{23}{100} \quad \text{for all } n \ge 5
\]

\subsection{Derivation of the Closed Form}
Let $T(x)$ represent the periodic tail starting at the 5th term.
\[
T(x) = \cfrac{-0.2288x}{1 + T(x)}
\]
Solving $T(x)^2 + T(x) + 0.2288x = 0$ yields the solution
\[
T(x) = \frac{-1 + \sqrt{1 - 0.915x}}{2}
\]
Defining $S = \sqrt{1 - 0.915x}$, the denominator term for the truncated fraction is $1 + T(x) = \frac{1+S}{2}$. We substitute this into the continued fraction truncated at $a_4$.
\[
A_{h2}(x) = 1 + \cfrac{a_1 x}{1 + \cfrac{a_2 x}{1 + \cfrac{a_3 x}{1 + \cfrac{a_4 x}{\frac{1+S}{2}}}}}
\]
We substitute the exact coefficients 
\[
a_1 = \frac{1}{4}, \quad a_2 = -\frac{1}{16}, \quad a_3 = -\frac{3}{16}, \quad a_4 = -\frac{3}{16}
\]
and simplify the nested fraction from the bottom up. \\

\paragraph{Level 4}
\[
1 + \frac{a_4 x}{(1+S)/2} = 1 - \frac{3x}{8(1+S)} = \frac{8(1+S)-3x}{8(1+S)}
\]

\paragraph{Level 3}
\[
\begin{aligned}
    1 + \frac{a_3 x}{\text{Level 4}} &= 1 - \frac{3x}{16} \cdot \frac{8(1+S)}{8(1+S)-3x} \\
    &= \frac{2(8(1+S)-3x) - 3x(1+S)}{2(8(1+S)-3x)} \\
    &= \frac{(16-9x)+(16-3x)S}{(16-6x)+16S}
\end{aligned}
\]

\paragraph{Level 2}
\[
\begin{aligned}
    1 + \frac{a_2 x}{\text{Level 3}} &= 1 - \frac{x}{16} \cdot \frac{2((8-3x)+8S)}{(16-9x)+(16-3x)S} \\
    &= \frac{(128-72x)+(128-24x)S - (8x-3x^2+8xS)}{(128-72x)+(128-24x)S} \\
    &= \frac{(128-80x+3x^2)+(128-32x)S}{(128-72x)+(128-24x)S}
\end{aligned}
\]

\noindent Finally, 
\paragraph{Level 1}
\[
\begin{aligned}
    1 + \frac{a_1 x}{\text{Level 2}} &= 1 + \frac{x}{4} \cdot \frac{(128-72x)+(128-24x)S}{(128-80x+3x^2)+(128-32x)S} \\
    &= 1 + \frac{x \cdot ((32-18x)+(32-6x)S)}{(128-80x+3x^2)+(128-32x)S} \\
    &= \frac{(128-80x+3x^2)+(128-32x)S + (32x-18x^2+(32x-6x^2)S)}{{(128-80x+3x^2)+(128-32x)S}} \\
    &= \frac{(128-48x-15x^2)+(128-6x^2)S}{{(128-80x+3x^2)+(128-32x)S}}
\end{aligned}
\]

\noindent We are able to derive the final rational form,
\[
A_{2}(x) = \frac{P(x) + Q(x)S}{R(x) + U(x)S},
\]
where $P(x), Q(x), Q(x), R(x)$ are
\[
\begin{aligned}
P(x) &= 128 - 48x - 15x^2 \\
Q(x) &= 128 - 6x^2 \\
R(x) &= 128 - 80x + 3x^2 \\
U(x) &= 128 - 32x
\end{aligned}
\]
Thus, the final approximation is
\[
\boxed{A_{2}(x) = \frac{(128 - 48x - 15x^2) + (128 - 6x^2)\sqrt{1-0.915x}}{(128 - 80x + 3x^2) + (128 - 32x)\sqrt{1-0.915x}}}
\]

\section{Experimental Results and Numerical Analysis}

To evaluate the practical performance of Ramanujan's original approximations and our derived perturbations, we conducted a numerical experiment over the full interval of convergence, $x \in [0, 1)$.\\

\noindent Using the double-precision floating-point arithmetic of python math libraries, we established a benchmark by computing the target series $E(x)$ by summing the first 50 terms.
\[ E_{50}(x) = \sum_{n=0}^{49} \left\{ \frac{1}{2n-1} \frac{1}{4^n} \binom{2n}{n} \right\}^2 x^n \]
Summing to 50 terms ensures that the truncation error is significantly smaller than standard machine epsilon ($ \approx 2.2 \times 10^{-16}$), providing an effectively exact ground truth for comparison.\\

Cantrell ~\cite{Cantrell2004} provides a different approximation that is worse than Ramanujan's expression for ellipses with lower eccentricity, but is better when the eccentricity gets closer to 1. His expression has the form

\begin{align*}
    P_{cantrell} &= 4(a + b) - 2(4-p) \frac{ab}{f} \\
    f &= p(a+b) + \frac{1-2p}{k+1}\sqrt{(a+kb)(ka+b)},
\end{align*}

\noindent where parameters $p$ and $k$ are jointly tuned. The best values prescribed are $k = 74$ and $p = 0.410$. Using techniques similar from those presented in section ~\ref{sec:sec3}, we can rewrite this formula as

\begin{align*}
    P_{cantrell} &= \pi (a+b) \left( \frac{4}{\pi} - \frac{4-\pi}{2} \frac{(1-x)}{(p + \frac{1-2p}{(k+1)} \sqrt{(\frac{k+1}{2})^2 - (\frac{k-1}{2})^2 x}} \right) \\
    &= \pi (a+b) C(x)
\end{align*}

\noindent In order to evaluate the quality of the various approximations, we compared the following five expressions.
\begin{enumerate}
    \item $R_1(x) = 3 - \sqrt{4-x}$ (Ramanujan-1)
    \item $R_2(x) = 1 + \frac{3x}{10 + \sqrt{4-3x}}$ (Ramanujan-2)
    \item $C(x) = \frac{4}{\pi} - \frac{4-\pi}{2} \frac{(1-x)}{(p + \frac{1-2p}{(k+1)} \sqrt{(\frac{k+1}{2})^2 - (\frac{k-1}{2})^2 x}}$ (Cantrell)
    \item $A_1(x) = 1 + \frac{192x}{640 + \sqrt{16384 - 12288x - 18x^4}}$ (Ours-1)
    \item $A_2(x) = \frac{(128 - 48x - 15x^2) + (128 - 6x^2)\sqrt{1-0.915x}}{(128 - 80x + 3x^2) + (128 - 32x)\sqrt{1-0.915x}}$ (Ours-2).
\end{enumerate}

\subsection{Visual Comparison}
Figure \ref{fig:comparison_plot} shows the direct plots of these four functions alongside the target $E(x)$. Due to the high quality of Ramanujan's initial insights, all four approximations follow the shape of the target curve exceptionally well. On a standard linear scale, the curves are nearly indistinguishable, highlighting that even the simplest approximation, $R_1(x)$, is sufficient for many low-precision applications.

\begin{figure}[htbp]
    \centering
    \includegraphics[width=\textwidth]{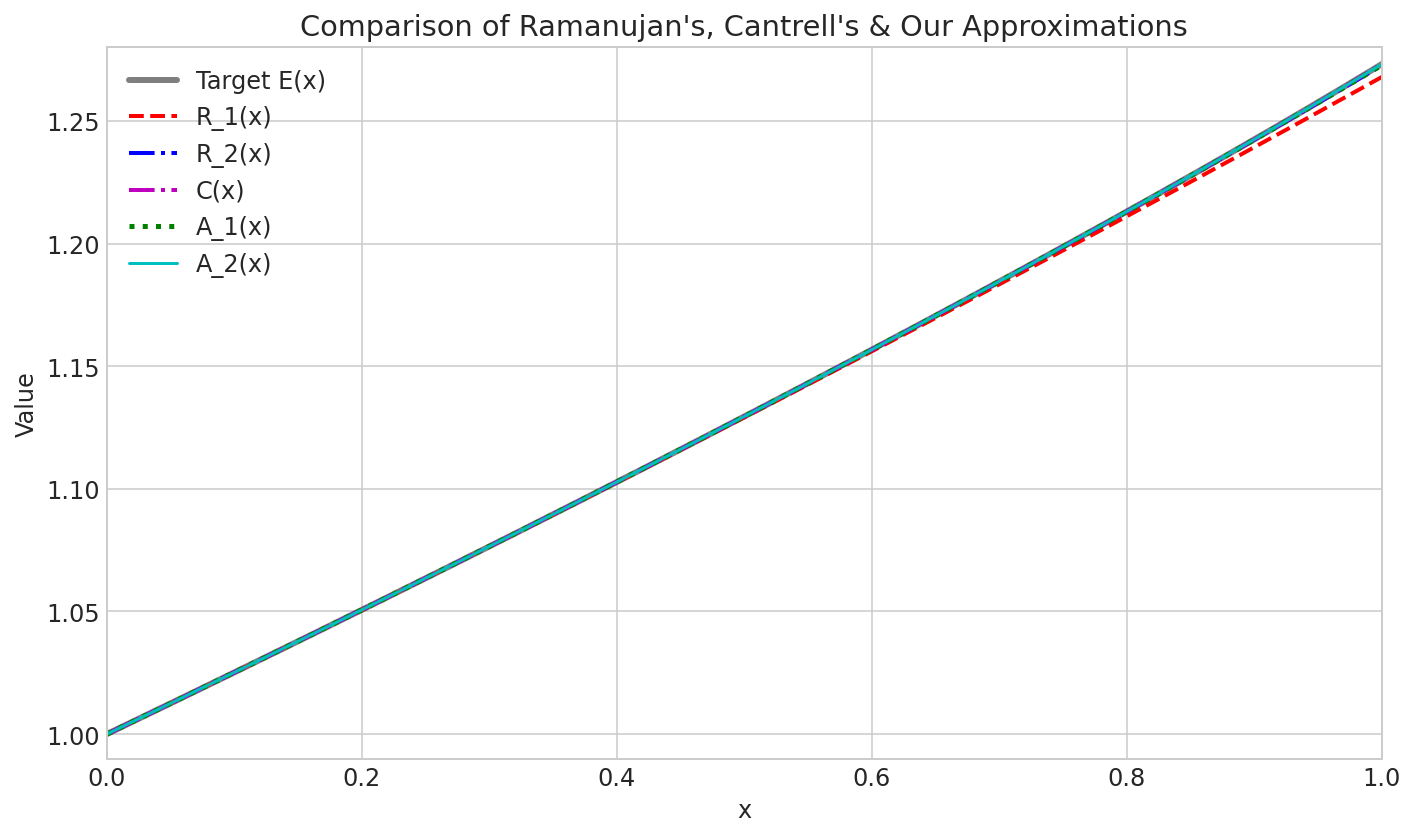}
    \caption{Direct comparison of Ramanujan's approximations $R_1(x)$ and $R_2(x)$, Cantrell's formula $C(x)$ and our approximations $A_1(x)$ and $A_2(x)$, against the target series $E(x)$ over the interval $[0,1)$. The remarkable accuracy of all functions makes them difficult to distinguish on this scale.}
    \label{fig:comparison_plot}
\end{figure}

\subsection{Relative Error Analysis}
To quantify the differences, we plotted the relative error, defined as $\frac{|\text{Approx}(x) - E(x)|}{E(x)}$, on a logarithmic scale in Figure \ref{fig:error_plot}. This plot clearly visualizes the distinct behaviors of the approximations.

\begin{figure}[htbp]
    \centering
    \includegraphics[width=\textwidth]{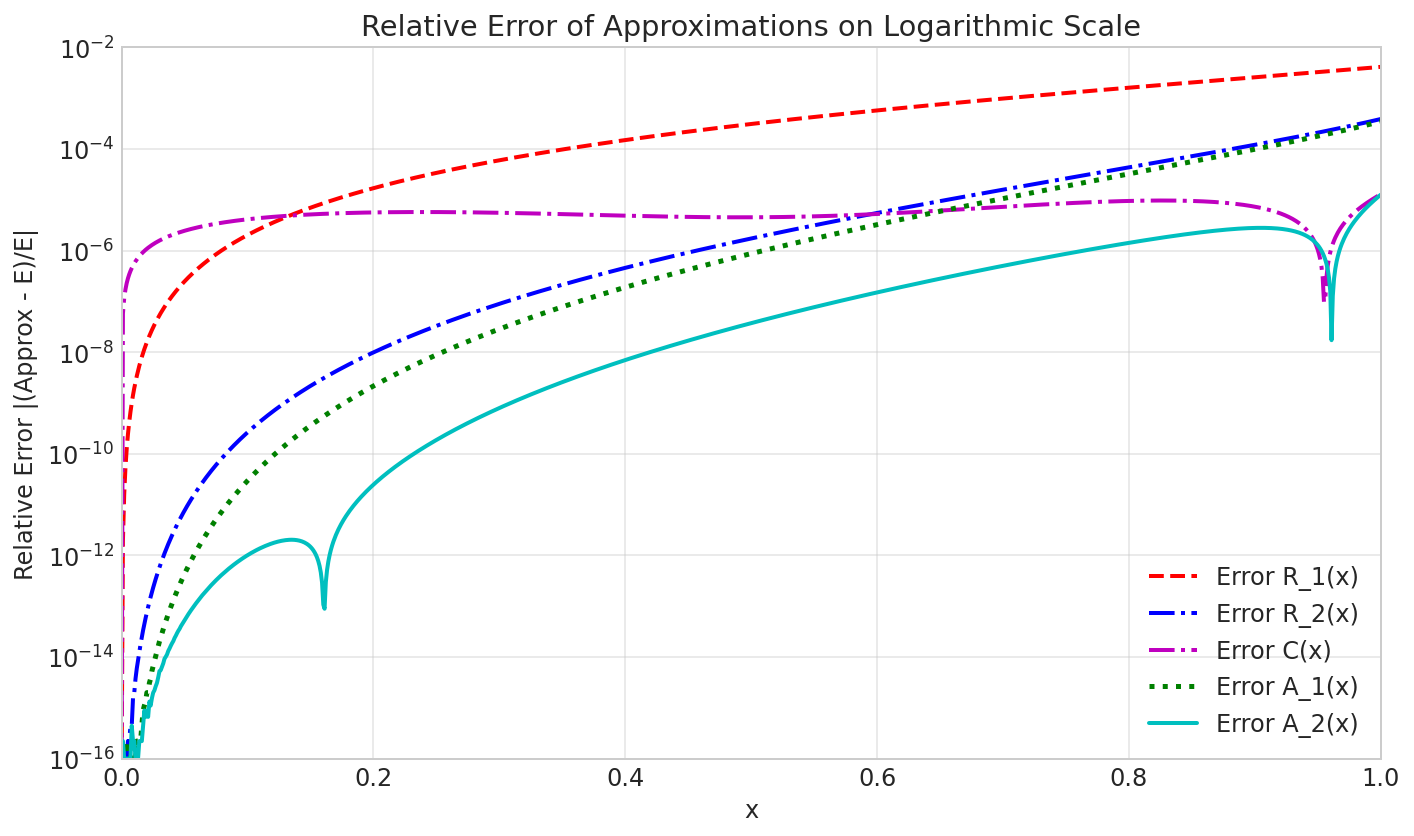}
    \caption{The relative error of the five approximations plotted on a logarithmic y-axis. The slopes near $x=0$ indicate the order of local accuracy. $A_2(x)$ demonstrates superior performance across the entire domain. Note that the y-axis is shown at log scale.}
    \label{fig:error_plot}
\end{figure}


\noindent In summary, while Ramanujan's $R_2(x)$ is an excellent general-purpose approximation, our approximation $A_2(x)$ is generated by approximating all the tail coefficients in the continued fraction expansion of $E(x)$ with a single value. Although not as elegant as Ramanujan's approximation, it provides vastly superior precision across the entire interval $[0,1)$.

\bibliographystyle{plain} 
\bibliography{main} 

\end{document}